\documentclass{article}
\usepackage{setspace}
\usepackage[left=1.5in,right=1.5in,top=1in,bottom=1in]{geometry}
\usepackage{amsmath,amssymb,amsthm}
\usepackage{mathbbol}
\usepackage{enumerate}
\usepackage{hyperref}
\usepackage{cleveref}
\usepackage{tikz-cd}

\numberwithin{equation}{section}

\theoremstyle{plain}
\newtheorem{theorem}{Theorem}[section]
\newtheorem{proposition}[theorem]{Proposition}
\newtheorem{corollary}[theorem]{Corollary}
\newtheorem{lemma}[theorem]{Lemma}
\theoremstyle{definition}
\newtheorem{definition}[theorem]{Definition}
\newtheorem{remark}[theorem]{Remark}

\newcommand\QQ{\mathbb{Q}}
\newcommand\ZZ{\mathbb{Z}}
\newcommand{\Hom}{\mathrm{Hom}}
\newcommand{\mc}{{\text{-}\mathrm{mC}_R}}
\DeclareMathOperator{\coker}{coker}

\author{Jiahao Hu}
\newcommand{\Addresses}{{%
  \bigskip
  \footnotesize

  \textsc{Yau Mathematical Sciences Center, Tsinghua University,
    Beijing, China 100084}\par\nopagebreak
  Email: \texttt{jiahao.hu.math@gmail.com}
}}
\date{}

\title{Quillen equivalences for truncated multicomplexes}
\begin{document}
\maketitle
\begin{abstract}
We establish Quillen equivalences between successive truncations of multicomplexes for the model structures that detect weak equivalences simultaneously on prescribed pages of the two associated spectral sequences. In particular, the natural comparison from $4$-multicomplexes to bicomplexes, which is relevant to the comparison between almost complex and complex geometry, is a Quillen equivalence. This answers a question of Cirici, Livernet, and Whitehouse.
\end{abstract}
\section{Introduction}

Multicomplexes arise naturally whenever a differential is decomposed into homogeneous components with respect to a grading. The equation that the differential squares to zero becomes precisely the system of quadratic relations defining a multicomplex, while the grading gives rise to an associated spectral sequence. Such decompositions occur in many geometric settings: the complexified de Rham complex of a complex manifold is naturally a bicomplex; $4$-multicomplexes occur in almost complex geometry \cite{cirici2021dolbeault} and, more generally, in the study of almost Dirac and generalized almost complex structures \cite{gualtieri2020deformation,cirici2025model}. Further examples arise in Morse--Bott theory \cite{banyaga2010morse,hurtubise2010multicomplexes} and in sub-Riemannian geometry, where multicomplexes encode the weight decomposition of the de Rham differential on Carnot groups \cite{lerario2023multicomplexes}. These geometric examples motivate the study of multicomplexes and their associated spectral sequences.

The homological and homotopical properties of multicomplexes have consequently been studied by several authors. Livernet, Whitehouse, and Ziegenhagen gave an explicit description of the differentials in the spectral sequence of a multicomplex \cite{livernet2020spectral}. Cirici, Egas Santander, Livernet, and Whitehouse constructed model structures on filtered complexes and bicomplexes in which weak equivalences are detected at a prescribed page of the associated spectral sequence \cite{cirici2020model}. Fu, Guan, Livernet, and Whitehouse subsequently extended these model structures from bicomplexes to multicomplexes and their finite truncations \cite{fu2020model}. More recently, Cirici, Livernet, and Whitehouse introduced model structures on truncated multicomplexes that treat the two naturally associated spectral sequences simultaneously \cite{cirici2025model}. Related model structures on right-half-plane bicomplexes and twisted complexes were constructed by Muro and Roitzheim \cite{muro2019homotopy}, while Livernet and Whitehouse developed a pagewise homotopy theory for spectral sequences themselves \cite{livernet2024homotopy}. Together, these works provide a homotopical framework in which multicomplexes arising from geometry can be compared.

The comparison between $4$-multicomplexes and bicomplexes is of particular geometric significance. On a complex manifold, the bigraded de Rham algebra is a bicomplex. On an almost complex manifold, the de Rham differential has two additional homogeneous components; these are determined by the Nijenhuis tensor and vanish precisely when the almost complex structure is integrable \cite{newlander1957complex}. After regrading and adjusting signs to match our conventions, the bigraded de Rham algebra is therefore a $4$-multicomplex in general and becomes a bicomplex in the integrable case. The passage from $4$-multicomplexes to bicomplexes studied here is thus the linear-algebraic counterpart of the passage from almost complex to complex geometry.

Cirici, Livernet, and Whitehouse formulated this comparison as a quotient adjunction from $4$-multicomplexes to bicomplexes. They proved that it is a Quillen adjunction and asked whether it is a Quillen equivalence \cite[Proposition~4.1.1 and Remark~4.1.2]{cirici2025model}. We answer their question as a special case of a general theorem for successive truncations.

To state this general result, we introduce the following notation. Let $R$ be a commutative ring with unit. For $N\geq2$, let $N\mc$ denote the category of $N$-multicomplexes over $R$; when $N=2$, this is the category of bicomplexes. Every $N$-multicomplex has two associated spectral sequences, denoted by ${}'E$ and ${}''E$. For $r,s\geq0$, let $(N\mc)_{r,s}$ denote the model structure whose weak equivalences are precisely the maps inducing quasi-isomorphisms on ${}'E_r$ and ${}''E_s$. For $N\geq3$, the two extreme structure operators determine two quotient adjunctions between successive truncations, denoted by $\mathbf q_+\dashv\mathbf j_+$ and $\mathbf q_-\dashv\mathbf j_-$. Our main result is the following.

\begin{theorem}
Let $R$ be a commutative unital ring, $N\ge 3$, and $r,s\ge 0$. The adjacent-truncation adjunctions are Quillen equivalences
\[
\begin{aligned}
\mathbf{q}_+:(N\mc)_{r,s+1}&\ \rightleftarrows\ \left((N-1)\mc\right)_{r,s}:\mathbf j_+,\\
\mathbf{q}_-: (N\mc)_{r+1,s}&\ \rightleftarrows\ \left((N-1)\mc\right)_{r,s}:\mathbf j_-.
\end{aligned}
\]
\end{theorem}

The page shifts in the theorem are forced by the elementary comparison of the two spectral sequences: $\mathbf j_+$ preserves the first spectral sequence and raises the page index of the second by one, whereas $\mathbf j_-$ preserves the second and raises the page index of the first by one. Specializing first to $\mathbf q_-$ for $N=4$ and then to $\mathbf q_+$ for $N=3$ yields the composite Quillen equivalence
\[
\mathbf{q}_+\mathbf{q}_-:(4\mc)_{1,1}\ \rightleftarrows\ (2\mc)_{0,0}:\mathbf j_-\mathbf j_+.
\]
Cirici, Livernet, and Whitehouse considered precisely this composite adjunction in \cite{cirici2025model}. The induced derived adjunction yields an equivalence between the corresponding homotopy categories of $4$-multicomplexes and bicomplexes.

The paper is organized as follows. Section~2 recalls the two spectral sequences associated with an $N$-multicomplex, their representing objects, and the model structures whose weak equivalences are detected on prescribed pages. Section~3 constructs the two base-change adjunctions between successive truncations and determines their effects on the representing objects and on both spectral sequences. Section~4 establishes the exactness properties needed to analyze cell attachments, verifies the adjunction unit on the basic cells, and proves the Quillen equivalences by cellular induction.

\section{Model structures}
In this section, we collect the notation and results from \cite{cirici2025model,cirici2020model,fu2020model} that are needed for stating and proving the desired Quillen equivalences. Throughout we work over a fixed commutative ring $R$ with unit.

\begin{definition}[$N$-multicomplex, {cf. \cite[Definition 2.1.1]{cirici2025model}}]
	A multicomplex is a bigraded $R$-module $X=\{X^{p,q}\}_{p,q\in \ZZ}$ equipped with $R$-module maps $\{\delta_i: X\to X\}_{i\ge 0}$ of bidegrees $|\delta_i|=(-i,1-i)$ satisfying that for all $l\ge 0$,
	\[
	\sum_{i+j=l}(-1)^i\delta_i\delta_j=0.
	\]
	For $N\ge 2$, an $N$-multicomplex is a multicomplex on which $\delta_i=0$ for all $i\ge N$. A morphism $f: X\to Y$ between ($N$-)multicomplexes is a degree (0,0) morphism satisfying that $\delta_i f=f\delta_i$ for all $i\ge 0$.
\end{definition}

Multicomplexes also appear under the name \emph{twisted complexes}; their $E_r$-homotopies were studied in \cite{cirici2018derived}.

	The category of $N$-multicomplexes, denoted by $N\mc$, can be viewed as the category of bigraded left modules over the bigraded associative $R$-algebra (\cite[Notation 2.1.2 and Lemma 2.1.4]{cirici2025model})
\[
\mathcal{A}_N=R\langle \delta_0,\dots,\delta_{N-1}\rangle/\left(\sum_{\substack{i+j=l\\0\le i,j<N}}(-1)^i\delta_i\delta_j\right)_{l\ge 0},
\]
where $R\langle \delta_0,\dots,\delta_{N-1}\rangle$ is the free associative $R$-algebra generated by the symbols $\delta_0,\dots,\delta_{N-1}$ of bidegrees $|\delta_i|=(-i,1-i)$.

\begin{remark}\label{contractible}
	Let $\mathcal{D}$ be the subalgebra of $\mathcal{A}_N$ generated by $\delta_0$. Then $\mathcal{D}\cong R\langle\delta_0\rangle/(\delta_0^2)$. We will need the fact (\cite[Lemma 2.1.4]{cirici2025model}) that $\mathcal{A}_N$ is free as a right $\mathcal{D}$-module, and that $\mathcal{D}=R\oplus R\delta_0$ carries a canonical contracting $\delta_0$-homotopy $h$ (i.e. $\delta_0h+h\delta_0=\mathrm{id}_\mathcal{D}$), given by $h(a+b\delta_0)=b$. See also \cite{auyeung2023algebra} for the special case $N=4$ and $R=\QQ$.
\end{remark}

\begin{definition}[witness cycles and boundaries, {cf. \cite[Definition 2.2.1]{cirici2025model}}]
	Let $X$ be an $N$-multicomplex. We define the \textit{witness $r$-cycles} to be the bigraded $R$-modules $ZW_0^{p,q}(X)=X^{p,q}$ and for $r\ge 1$,
	\begin{align*}
		ZW_r^{p,q}(X)=\{(x_0,\dots,x_{r-1})|&\text{$x_j\in X^{p-j, q-j}$ for $0\le j< r$ such that}\\
		&\text{$\sum_{i+j=l}(-1)^i \delta_i x_j=0$ for $0\le l< r$}\}.
	\end{align*}
	Define the \textit{witness $r$-boundaries} to be the bigraded $R$-modules $BW_0^{p,q}(X)=0$, $BW_1^{p,q}(X)=X^{p,q}$ and for $r\ge 2$,
	\[
	BW_r^{p,q}(X)=ZW_{r-1}^{p+r-1,q+r-1}(X)\oplus X^{p,q}\oplus ZW_{r-1}^{p-1,q}(X).
	\]
	Define the bidegree (0,1) map $w_r: BW_r^{p,q-1}(X)\to ZW_r^{p,q}(X)$ by $w_0=0$, $w_1=\delta_0$ and for $r\ge 2$,
	\[
	w_r:\left(b_0,\dots,b_{r-2};a;c_0,\dots,c_{r-2}\right)\mapsto \left(\delta_j a+\sum_{i=1}^{r-1}(-1)^i\delta_{i+j} b_{r-1-i}+c_{j-1}\right)_{0\le j<r}
	\]
	in which $c_{-1}=0$. Finally define $E_r^{p,q}(X)=ZW_r^{p,q}/w_r BW_r^{p,q-1}(X)$.
\end{definition}

From \cite[Lemma 3.18]{fu2020model} the functors $ZW_r^{p,q}(-)$ and $BW_r^{p,q}(-)$ are represented by $N$-multicomplexes $\mathcal{ZW}_r^N(p,q)$, $\mathcal{BW}_r^N(p,q)$ respectively. The map $w_r$ is represented by a morphism of $N$-multicomplexes
\[
\iota_r^N(p,q): \mathcal{ZW}_r^N(p,q)\to \mathcal{BW}_r^N(p,q-1).
\]
For $r\ge 1$ the kernel of $w_r$ is naturally isomorphic to $ZW_r^{p+r-1,q+r-2}(X)$ by \cite[Lemma 3.8]{fu2020model} and therefore we have an exact sequence (\cite[Lemma 3.20]{fu2020model})
\begin{equation}\label{es}
	\mathcal{ZW}_r^N(p,q)\xrightarrow{\iota_r^N(p,q)} \mathcal{BW}_r^N(p,q-1)\to \mathcal{ZW}_r^N(p+r-1, q+r-2)\to 0
\end{equation}

The standard filtration of the total complex of $X$ gives an associated spectral sequence; see, for example, \cite[Section~11]{boardman1999conditionally}. By \cite[Proposition~3.7]{fu2020model}, the quotient defined above is its $r$-th page. After translating the signs and bidegrees to our conventions, the differential $d_r:E_r^{p,q}(X)\to E_r^{p-r,q+1-r}(X)$ is given by \cite[Theorem~2.10]{livernet2020spectral}: $d_0=\delta_0$, and for $r\geq1$,
\begin{equation}\label{page-differential}
	d_r[x_0, \dots, x_{r-1}]=[\sum_{i=1}^r(-1)^i\delta_{i+j}x_{r-i}]_{0\le j<r}.
\end{equation}
In general, the page differential $d_r$ is not induced by the single structure map $\delta_r$; explicit examples illustrating this distinction were given by Hurtubise \cite{hurtubise2010multicomplexes}.

The involution $\mathfrak{o}_N:\ZZ^2\to\ZZ^2$ given by
\[
\mathfrak{o}_N(p,q)=\left((N-2)p-(N-1)q, (N-3)p-(N-2)q\right)
\]
yields an involution of $N\mc$ given on objects by $(X^{\mathfrak{o}_N})^{p,q}=X^{\mathfrak{o}_N(p,q)}$, $\delta_i^{\mathfrak{o}_N}=\delta_{N-1-i}$ and on morphisms by $(f^{\mathfrak{o}_N})^{p,q}=f^{\mathfrak{o}_N(p,q)}$. See \cite[Notation 2.1.5 and Lemma 2.1.6]{cirici2025model}. Consequently we obtain another spectral sequence $E_r(X^{\mathfrak{o}_N})$ associated to an $N$-multicomplex $X$.

Henceforth, we denote the spectral sequences $E_r(X)$, $E_r(X^{\mathfrak{o}_N})$ by $'E_r(X)$, $''E_r(X)$ and call them \textit{the first and the second spectral sequences} respectively. The differentials are denoted by $'d_r$ and $''d_r$ accordingly. These spectral sequences are computed for the representing $N$-multicomplexes.
\begin{lemma}[{\cite[Corollary 6.4]{fu2020model},\cite[Proposition 3.2.4]{cirici2025model}}]\label{computation}
	For every $p,q\in \ZZ$, $r\ge 0$ and $i\ge 1$ we have
	\begin{align*}
		'E_i(\mathcal{ZW}_r^N(p,q)) & \cong\begin{cases}
			R^{p,q}\oplus R^{p-r,q+1-r} &\hbox{if $1\le i\le r$}\\
			0 &\hbox{if $i\ge r+1$}
		\end{cases}\\
		''E_i(\mathcal{ZW}_r^N(p,q))&=0.
	\end{align*}
	Here $R^{m,n}$ stands for the free $R$-module of rank one concentrated in bidegree $(m,n)$.
\end{lemma}
\begin{remark}\label{generator}
	The component $'E_r^{p,q}(\mathcal{ZW}_r^N(p,q))=R$ is generated by the tautological  class $\xi_r$ represented by $\mathrm{id}_{\mathcal{ZW}}\in \Hom_{N\mc}\left(\mathcal{ZW}_r^N(p,q), \mathcal{ZW}_r^N(p,q)\right)=ZW_r^{p,q}\left(\mathcal{ZW}_r^N(p,q)\right)$. Indeed let $R(p,q)$ be the $N$-multicomplex with $R(p,q)^{i,j}=R$ if $(i,j)=(p,q)$ and $0$ otherwise. The differentials are forced to vanish on $R(p,q)$. Then $1\in R= R(p,q)^{p,q}$ is a witness $r$-cycle and therefore corresponds to a map
	\[
	\varepsilon_r:\mathcal{ZW}_r^N(p,q)\to R(p,q).
	\]
	This induces a map $R={'E}_r^{p,q}(\mathcal{ZW}_r^N(p,q))\to{'E}_r^{p,q}(R(p,q))=R$ which by construction takes $\xi_r$ to $1\in R$. This proves that the tautological class $\xi_r$ is a generator. Also since $'E_{r+1}(\mathcal{ZW}_r^N(p,q))=0$, the class $'d_r\xi_r$ must be the other generator of $'E_r(\mathcal{ZW}_r^N(p,q))$ for $r\ge 1$.
\end{remark}

The spectral sequences and the multicomplexes $\mathcal{ZW}$, $\mathcal{BW}$ are used to build model structures on $N\mc$.
\begin{theorem}[{\cite[Theorem 3.1.2]{cirici2025model}}]
	For every $r,s\ge 0$, the category $N\mc$ admits a cofibrantly generated model structure $(N\mc)_{r,s}$ in which
	\begin{enumerate}[(1)]
		\item Weak equivalences are morphisms inducing quasi-isomorphisms on $'E_r$ and $''E_s$ (i.e. isomorphisms on $'E_{r+1}$ and $''E_{s+1}$).
		\item Fibrations are morphisms $f$ such that $'E_i(f)$ and $''E_j(f)$ are bidegreewise surjective for every $0\le i\le r$ and $0\le j\le s$.
		\item Cofibrations are generated by
		\begin{enumerate}
			\item $\iota_{r+1}^N(p,q)$ and $(\iota_{s+1}^N(p,q))^{\mathfrak{o}_N}$, together with
			\item $0\to \mathcal{ZW}_i^N(p,q)$ and $0\to (\mathcal{ZW}_j^N(p,q))^{\mathfrak{o}_N}$ for $1\le i< r$ and $1\le j< s$.
		\end{enumerate}
	\end{enumerate}
	Moreover, the model category $(N\mc)_{r,s}$ is right proper, combinatorial and tractable.
\end{theorem}
\begin{remark}\label{involution-equivalence}
	Since the involution $\mathfrak{o}_N$ interchanges the two spectral sequences, it induces an isomorphism
	\[
	(-)^{\mathfrak{o}_N}:(N\mc)_{r,s}\simeq (N\mc)_{s,r}
	\]
	of  model categories. See also \cite[Remark 3.4.3]{cirici2025model}.
\end{remark}

\section{Base-change adjunctions}
The natural bigraded $R$-algebra morphism $\phi: \mathcal{A}_N\to \mathcal{A}_{N-1}$ sending $\delta_i$ to $\delta_i$ for $0\le i<N-1$ and $\delta_{N-1}$ to $0$ yields an adjunction
\[
\mathbf{q}_+: N\mc\leftrightarrows (N-1)\mc: \mathbf{j}_+
\]
by base-changes along $\phi$. Explicitly, we have
\begin{equation}\label{j+}
	(\mathbf{j}_+Y)^{p,q}=Y^{p,q},\quad \delta_i^{\mathbf{j}_+Y}=\begin{cases}
	\delta_i^Y &\hbox{$0\le i<N-1$}\\
	0 &\hbox{$i=N-1$}
\end{cases}
\end{equation}
Using the involutions $\mathfrak{o}_N$ and $\mathfrak{o}_{N-1}$, we obtain another pair of adjoint functors $\mathbf{q}_-\dashv \mathbf{j}_-$ by
\[
\mathbf{q}_-=(-)^{\mathfrak{o}_{N-1}}\circ \mathbf{q}_+\circ (-)^{\mathfrak{o}_N}, \quad \mathbf{j}_-=(-)^{\mathfrak{o}_N}\circ \mathbf{j}_+\circ(-)^{\mathfrak{o}_{N-1}}.
\]
Unwinding the definition, we see that
\begin{equation}\label{j-}
	(\mathbf{j}_-Y)^{p,q}=Y^{\tau(p,q)},\quad \delta_i^{\mathbf{j}_-Y}=\begin{cases}
	0 &\hbox{$i=0$}\\
	\delta_{i-1}^Y &\hbox{$0< i\le N-1$}	
\end{cases}
\end{equation}
in which $\tau(p,q)=(q,2q-p)$. Note that $\tau=\mathfrak{o}_{N-1}\circ\mathfrak{o}_N$ and $\tau\left(|\delta_i^{\mathbf{j}_-Y}|\right)=|\delta_{i-1}^Y|$ for $1\le i\le N-1$.

\begin{proposition}\label{witness-comparison} Assume $N\ge 3$ and let $Y$ be an $(N-1)$-multicomplex.
\begin{enumerate}[(1)]
	\item For $r\ge 0$, there are natural isomorphisms
	\begin{align*}
		ZW^{p,q}_r(\mathbf{j}_+Y)&\cong ZW_r^{p,q}(Y),\\
		BW_r^{p,q-1}(\mathbf{j}_+Y)&\cong BW_r^{p,q-1}(Y),
	\end{align*}
	under which $w_r^{\mathbf{j}_+Y}=w_r^Y$.
	\item For simplicity, write $p'=q$ and $q'=2q-p$. There are natural isomorphisms
	\begin{align*}
		ZW^{p,q}_r(\mathbf{j}_-Y)&\cong\begin{cases}
		Y^{p',q'} &\hbox{$r=0,1$}\\
		ZW_{r-1}^{p',q'}(Y)\oplus Y^{p'-r+1, q'-r+1}&\hbox{$r\ge 2$}
		\end{cases},\\
		BW_r^{p,q-1}(\mathbf{j}_-Y)&\cong\begin{cases}
			0&\hbox{$r=0$}\\
			Y^{p'-1,q'-2}&\hbox{$r=1$}\\
			BW_{r-1}^{p',q'-1}(Y)\oplus Y^{p'-1, q'-2}\oplus Y^{p'-r+1, q'-r+1}&\hbox{$r\ge 2$}
		\end{cases}.
	\end{align*}
	Under these isomorphisms $w_0^{\mathbf{j}_- Y}=0$, $w_1^{\mathbf{j}_- Y}=0$ and for $r\ge 2$,
	\[
	w_r^{\mathbf{j}_-Y}=
	\begin{pmatrix}
		w_{r-1}^Y & 0 & 0\\
		0 & 0 & \mathrm{id}
	\end{pmatrix}.
	\]
\end{enumerate}
	\end{proposition}
\begin{proof}
The isomorphisms for $\mathbf{j}_+Y$ follow immediately from the description \labelcref{j+}.	For $\mathbf{j}_-Y$, the isomorphisms in the cases $r=0,1$ follow immediately from the description \labelcref{j-}. Assume now $r\ge 2$.

By definition $(x_0,\dots,x_{r-1})\in ZW_r^{p,q}(\mathbf{j}_- Y)$ if and only if $x_j\in (\mathbf{j}_-Y)^{p-j,q-j}=Y^{p'-j,q'-j}$ for $0\le j\le r-1$ and
	\[
	0=\sum_{i+j=l}(-1)^i\delta_i^{\mathbf{j}_-Y}(x_j)=-\sum_{i+j=l, i\ge 1}(-1)^{i-1}\delta_{i-1}^Y(x_j)
	\]
	for $0\le l\le r-1$.
	Observe that $x_{r-1}\in Y^{p'-r+1, q'-r+1}$ does not appear in the equations above, and that $(x_0,\dots,x_{r-2})$ is an element of $ZW_{r-1}^{p',q'}(Y)$. So we obtain an isomorphism
	\begin{align*}
		ZW^{p,q}_r(\mathbf{j}_-Y)&\xrightarrow{\simeq} ZW_{r-1}^{p',q'}(Y)\oplus Y^{p'-r+1, q'-r+1}\\
		(x_0,\dots,x_{r-1})&\mapsto \left(\left(x_0,\dots,x_{r-2}\right), x_{r-1}\right).
	\end{align*}
	
	Given $\beta=(b_0, \dots, b_{r-2}; a; c_0,\dots,c_{r-2})\in BW^{p,q-1}_r(\mathbf{j}_-Y)$, we write $w_r^{\mathbf{j}_-Y}(\beta)=(x_0,\dots,x_{r-1})$
	and define
	\[
	\overline{b}_j=-b_j,\ \overline{c}_j=c_j+\delta_j^Y a\ (0\le j<r-2);\quad \overline{a}=-b_{r-2},\ v=x_{r-1}.
	\]
	Since $\sum\limits_{i+j=l}(-1)^i\delta_i^Y\delta_j^Y a=0$ for every $l$, the preceding argument shows that 
	\[
	\overline{\beta}=(\overline{b}_0,\dots,\overline{b}_{r-3}; \overline{a};\overline{c}_0,\dots,\overline{c}_{r-3})\in BW_{r-1}^{p',q'-1}(Y).
	\]
	Now we claim the map
	\begin{align*}
		BW_{r}^{p,q-1}(\mathbf{j}_-Y)&\to BW_{r-1}^{p',q'-1}(Y)\oplus Y^{p'-1, q'-2}\oplus Y^{p'-r+1, q'-r+1}\\
		\beta &\mapsto (\overline{\beta};a; v)
	\end{align*}
	gives the required isomorphism. Indeed, this map is invertible because all the components of $\beta$ can be recovered from $(\overline{\beta},a)$ except for $c_{r-2}$, and then
	\[
	v=x_{r-1}=\delta_{r-2}^{Y} a+\sum_{i=1}^{r-1}(-1)^i\delta_{i+r-2}^{Y} b_{r-1-i}+c_{r-2}
	\]
	determines $c_{r-2}$. Moreover, for $0\le j<r-1$ we have
	\begin{align*}
		x_j &=\delta_{j-1}^Y a+\sum_{i=1}^{r-1}(-1)^i\delta_{i+j-1}^{Y} b_{r-1-i}+c_{j-1}\quad{\text{(we set $c_{-1}=\delta_{-1}^Y a=0$)}}\\
		&=-\delta_j^Y b_{r-2}+\sum_{i=2}^{r-1}(-1)^i\delta_{i+j-1}^{Y} b_{r-1-i}+ (c_{j-1}+\delta_{j-1}^Y a)\\
		&=\delta_j^Y\overline{a}+\sum_{i=1}^{r-2}(-1)^i\delta_{i+j}^Y \overline{b}_{r-2-i}+\overline{c}_{j-1}.
	\end{align*}
	This is the corresponding component of $w_{r-1}^{Y}(\overline{\beta})$. Together with $v=x_{r-1}$, this proves the asserted matrix form of $w_r^{\mathbf{j}_-Y}$.
\end{proof}

\begin{corollary}\label{basechange}
	Assume $N\ge 3$ and $p,q\in \ZZ$.
	\begin{enumerate}[(1)]
		\item For $r\ge 0$ there are isomorphisms of $(N-1)$-multicomplexes
		\begin{align*}
		\mathbf{q}_+\mathcal{ZW}_r^N(p,q)&\cong\mathcal{ZW}_r^{N-1}(p,q),\\
		\mathbf{q}_+\mathcal{BW}_r^N(p,q-1)&\cong\mathcal{BW}_r^{N-1}(p,q-1),
	\end{align*}
	under which $\mathbf{q}_+\left(\iota_r^N(p,q)\right) = \iota_r^{N-1}(p,q)$.
		\item Write $(p',q')=(q,2q-p)$. There are isomorphisms of $(N-1)$-multicomplexes
		\begin{align*}
			\mathbf{q}_-\mathcal{ZW}_0^N(p,q)\cong \mathcal{ZW}_0^{N-1}(p',q'),&\quad\mathbf{q}_-\mathcal{BW}_0^N(p,q-1)=0,\\
			\mathbf{q}_-\mathcal{ZW}_1^N(p,q)\cong \mathcal{ZW}_0^{N-1}(p',q'),&\quad\mathbf{q}_-\mathcal{BW}_1^N(p,q-1)\cong\mathcal{ZW}_0^{N-1}(p'-1,q'-2),
		\end{align*}
		under which $\mathbf{q}_-\iota_r^{N}(p,q)=0$ for $r=0,1$.

For $r\ge 2$, there are isomorphisms of $(N-1)$-multicomplexes
\begin{align*}
	\mathbf{q}_-\mathcal{ZW}_r^N(p,q)&\cong \mathcal{ZW}_{r-1}^{N-1}(p',q')\oplus C,\\
	\mathbf{q}_-\mathcal{BW}_r^N(p,q-1)&\cong \mathcal{BW}_{r-1}^{N-1}(p',q'-1)\oplus V\oplus C,
\end{align*}
in which $C=\mathcal{ZW}_0^{N-1}(p'-r+1, q'-r+1)$, $V=\mathcal{ZW}_0^{N-1}(p'-1, q'-2)$. Moreover under these isomorphisms we have
	\[
	\mathbf{q}_-\iota_r^{N}(p,q)=
	\begin{pmatrix}
		\iota_{r-1}^{N-1}(p',q')&0\\
  		0&0\\
  		0&\mathrm{id}_C
	\end{pmatrix}.
	\]
	\end{enumerate}
	\end{corollary}
	\begin{proof}
		This follows by the previous proposition from the adjunctions $\mathbf{q}_\pm\dashv\mathbf{j}_{\pm}$ and the Yoneda lemma.
	\end{proof}

\begin{proposition}\label{page-comparison} Assume $N\ge 3$ and let $Y$ be an $(N-1)$-multicomplex. Then there are natural isomorphisms of $R$-modules
	\begin{align*}
		'E_r^{p,q}(\mathbf{j}_+Y)&\cong {'E}_r^{p,q}(Y),\ r\ge 0\\
		''E_s^{p,q}(\mathbf{j}_+Y)&\cong\begin{cases}
			''E_{s-1}^{p',q'}(Y) &\hbox{$s\ge 1$} \\
			Y^{{\mathfrak{o}_N}(p,q)} &\hbox{$s=0$}
		\end{cases} 
	\end{align*}
	and
	\begin{align*}
		'E_r^{p,q}(\mathbf{j}_-Y)&\cong\begin{cases}
			'E_{r-1}^{p',q'}(Y) &\hbox{$r\ge 1$} \\
			Y^{p',q'} &\hbox{$r=0$}
		\end{cases}\\
		''E_s^{p,q}(\mathbf{j}_-Y)&\cong {''E}_{s}^{p,q}(Y),\ s\ge 0.
	\end{align*}
	Moreover the isomorphisms are compatible with the respective differentials up to sign. Here $(p',q')=(q,2q-p)$ and the differential on $Y^{{\mathfrak{o}_N}(p,q)}$ is zero.
\end{proposition}
\begin{proof} By applying the involutions $\mathfrak{o}_N, \mathfrak{o}_{N-1}$, it suffices to prove the proposition for $\mathbf{j}_+$.
	
	The assertions for the first spectral sequence follow directly from \Cref{witness-comparison}. For the second spectral sequence, we note that by definition
	\[
	''E_s^{p,q}(\mathbf{j}_+Y)={'E}_s^{p,q}\left((\mathbf{j}_+Y)^{\mathfrak{o}_N}\right)={'E}_s^{p,q}(\mathbf{j}_-(Y^{\mathfrak{o}_{N-1}})).
	\]
	Then by \Cref{witness-comparison}, for $s\ge 2$ we have (writing $\mathfrak{o}$ for $\mathfrak{o}_{N-1}$ for simplicity)
	\begin{align*}
		{'E}_s^{p,q}(\mathbf{j}_-(Y^{\mathfrak{o}}))&=\coker^{p,q}\left(w_s^{\mathbf{j_-}(Y^{\mathfrak{o}})}\right)\\
		&\cong\coker^{p',q'}\left(w_{s-1}^{Y^{\mathfrak{o}}}\right)={'E_{s-1}^{p',q'}}(Y^{\mathfrak{o}})={''E}_{s-1}^{p',q'}(Y).
	\end{align*}
	Now for the differential, write $Z=Y^{\mathfrak{o}_{N-1}}$ and let $(x_0,\dots,x_{s-1})\in ZW_{s}^{p,q}(\mathbf{j}_-Z)$. Then from \labelcref{page-differential} we have
	\[
		'd_s[x_0,\dots,x_{s-1}]=[\sum_{i=1}^s(-1)^i\delta_{i+j}^{\mathbf{j}_-Z} x_{s-i}]_{0\le j<s}=[-\delta_j^Z x_{s-1}-\sum_{i=1}^{s-1}(-1)^i\delta_{i+j}^Z x_{s-1-i}]_{0\le j<s}.
	\]
	Notice that $(\delta_j^Z x_{s-1})_{0\le j<s-1}=w_{s-1}^Z(0;x_{s-1};0)$ (or $w_1(x_1)$ if $s=2$) is the image of a witness boundary. Therefore under the isomorphism ${'E}_s^{p,q}(\mathbf{j}_-Z)\cong {'E}_{s-1}^{p',q'}(Z)$, $'d_s[x_0,\dots,x_{s-1}]$ is mapped to
	\begin{align*}
		[-\delta_j^Z x_{s-1}-\sum_{i=1}^{s-1}(-1)^i\delta_{i+j}^Z x_{s-1-i}]_{0\le j<s-1}&=-[\sum_{i=1}^{s-1}(-1)^i\delta_{i+j}^Z x_{s-1-i}]_{0\le j<s-1}\\
		&=-{'d}_{s-1}[x_0,\dots,x_{s-2}].
	\end{align*}
This proves the assertions for the second spectral sequence for $s\ge 2$. The case $s=1$ is similar and the case $s=0$ is straightforward.
\end{proof}


\begin{theorem}
	For $N\ge 3$ and $r,s\ge 0$, the adjunctions $\mathbf{q}_\pm\dashv\mathbf{j}_\pm$ give Quillen adjunctions
\begin{align*}
	& \mathbf{q}_+: (N\mc)_{r,s+1} \leftrightarrows \left((N-1)\mc\right)_{r,s} : \mathbf{j}_+\\
	& \mathbf{q}_-: (N\mc)_{r+1,s} \leftrightarrows \left((N-1)\mc\right)_{r,s} : \mathbf{j}_-
\end{align*}
\end{theorem}
\begin{proof}
	From the isomorphisms of the spectral sequences in \Cref{page-comparison}, we see that $\mathbf{j}_+$ preserves fibrations and weak equivalences. Therefore $\mathbf{q}_+\dashv\mathbf{j}_+$ is a Quillen adjunction. That $\mathbf{q}_-\dashv\mathbf{j}_-$ is also a Quillen adjunction follows from \Cref{involution-equivalence}.
\end{proof}

\section{Quillen equivalences}
\begin{proposition}\label{monomorphism}
	For $r,s\ge 0$, every cofibration in $(N\mc)_{r,s}$ is a monomorphism.
\end{proposition}
\begin{proof}
	Since $N\mc$ is abelian and colimits are computed bidegreewise, monomorphisms are stable under coproducts and pushouts. The exactness of filtered colimits of $R$-modules gives stability under transfinite compositions, and monomorphisms are also closed under retracts. Since every cofibration is a retract of a relative cell complex generated by the generating cofibrations \cite[Proposition~2.1.18]{hovey2007model}, it suffices to prove that the generating cofibrations are monomorphisms. We are therefore reduced to proving that $\iota_r^N(p,q)$ is injective for $r\ge 1$.
	
	Let $U:N\mc\to \mathrm{bgMod}_R$ be the forgetful functor from $N$-multicomplexes to the category of bigraded $R$-modules. It has a right adjoint $M\mapsto GM$ given by
	\[
	GM^{p,q}=\underline{\Hom}_R^{p,q}(\mathcal{A}_N,M)=\prod_{m,n} \Hom_R(\mathcal{A}_N^{m,n}, M^{m+p,n+q}),
	\]
	on which the left $\mathcal{A}_N$ action is induced by the right $\mathcal{A}_N$-multiplication on itself. Recall from \Cref{contractible} that $\mathcal{A}_N$ is a free right $\mathcal{D}$-module, so $GM$ as a left $\mathcal{D}$-module is a product of shifts of $\underline{\Hom}_R(\mathcal{D},M)$ which is $\delta_0$-acyclic (because it has an induced contracting homotopy). Therefore $'E_r(GM)=0$ for all $r\ge 1$. The adjunction unit $e_X: X\to GUX$ given by $e_X(x)(a)=ax$ is injective: evaluation at $1\in \mathcal{A}_N$ is a left inverse.
	
	Now consider $Z=\mathcal{ZW}_r^{N}(p,q)$ for $r\ge 1$. The morphism $e_Z$ represents a class in $ZW_r^{p,q}(GUZ)$, which must be in the image of $w_r$ because $E_r(GUZ)=0$. This means that there exists $f:\mathcal{BW}_r^N(p,q-1)\to GUZ$ such that $e_Z=f\circ\iota_{r}^N(p,q)$. Since $e_Z$ is injective, so is $\iota_r^N(p,q)$.
\end{proof}

\begin{corollary}\label{ses}
	For $r\ge 1$, the sequence \labelcref{es} is part of a short exact sequence:
	\[
	0\to \mathcal{ZW}_r^N(p,q)\xrightarrow{\iota_r^N(p,q)} \mathcal{BW}_r^N(p,q-1)\to \mathcal{ZW}_r^N(p+r-1, q+r-2)\to 0.
	\]
\end{corollary}
\begin{proof}
	The previous proposition shows that $\iota_r^N(p,q)$ is injective.
\end{proof}

\begin{corollary}\label{exactness}
	Let $i: A\to B$ be a cofibration in $(N\mc)_{r,s}$ with $C=\coker(i)$, then the sequence
	\[
	0\to A\xrightarrow{i} B\to C\to 0
	\]
	is exact, and if $N\ge 3$ the induced sequences
	\[
	0\to \mathbf{q}_\pm A\xrightarrow{\mathbf{q}_\pm i}\mathbf{q}_\pm B\to \mathbf{q}_\pm C\to 0
	\]
	are exact. Here we assume $s\ge 1$ for $\mathbf{q}_+$ and assume $r\ge 1$ for $\mathbf{q}_-$.
\end{corollary}
\begin{proof}
	By \Cref{monomorphism}, $i$ is injective and therefore the sequence
	\[
	0\to A\xrightarrow{i} B\to C\to 0
	\]
	is exact. On the one hand, $\mathbf{q}_\pm$ are left Quillen and $i$ is a cofibration, so $\mathbf{q}_\pm i$ are cofibrations and thus injective by \Cref{monomorphism}. On the other hand, $\mathbf{q}_\pm$ are left adjoints and therefore right exact. So $\mathbf{q}_\pm$ turn the short exact sequence above into short exact sequences.
\end{proof}

\begin{lemma}\label{cell-attachment}
	Let $f:\mathcal{ZW}_2^N(p,q)\to X$ be a morphism of $N$-multicomplexes, and form the pushout
	\[
	\begin{tikzcd}
		\mathcal{ZW}_2^N(p,q)\ar[r,"f"]\ar[d,"{\iota_2^N(p,q)}"'] & X\ar[d,"i"]\\
		\mathcal{BW}_2^N(p,q-1)\ar[r] & Y
	\end{tikzcd}
	\]
	Then $i$ is injective with $Z=\coker(i)\cong \mathcal{ZW}_2^N(p+1,q)$. Moreover, the  induced sequence
	\[
	0\to {'E}_1(X)\xrightarrow{i_*}  {'E}_1(Y)\to {'E}_1(Z)\to 0
	\]
	is exact.
\end{lemma}
\begin{proof}
	The first assertion follows from \Cref{monomorphism} and \Cref{ses}. Recall that $'E_1=H(-,\delta_0)$, so the desired short exact sequence will follow from the vanishing of the connecting homomorphism in the homology long exact sequence. For this, it suffices to prove the connecting homomorphism in the universal case $X=\mathcal{ZW}_2^N(p,q)$ and $i=\iota_2^N(p,q)$ vanishes because the general connecting homomorphism is the universal one followed by ${'E}_1(f)$. By \Cref{computation}, in the universal case $'E_1(X)=R^{p,q}\oplus R^{p-2,q-1}$ and ${'E}_1(Z)=R^{p+1,q}\oplus R^{p-1,q-1}$ but the connecting homomorphism has bidegree $(0,1)$ and thus vanishes for degree reasons.
\end{proof}

\begin{lemma}\label{unit-computation} Assume $N\ge 3$.
	Let $X$ be $\mathcal{ZW}_2^N(p,q)$ or $\mathcal{ZW}_1^N(p,q)^{\mathfrak{o}_N}$, and let $\eta_X: X\to \mathbf{j}_-\mathbf{q}_-X$ be the unit of the adjunction $\mathbf{q}_-\dashv\mathbf{j}_-$. Then $\eta_X$ is a weak equivalence in $(N\mc)_{1,0}$.
\end{lemma}
\begin{proof} We show that $\eta_X$ induces isomorphisms on $'E_2$ and ${''E}_1$.
	
	For $X=\mathcal{ZW}_2^N(p,q)$, from \Cref{basechange} we have $$\mathbf{q}_-X\cong \mathcal{ZW}_1^{N-1}(p',q')\oplus \mathcal{ZW}_0^{N-1}(p'-1, q'-1)$$ where $(p',q')=\tau(p,q)$. Let $s: \mathcal{ZW}_1^{N-1}(p',q')\to \mathbf{q}_-X$ be the inclusion of the first factor. We remark that the second factor does not contribute to $'E_1$ by \Cref{computation}. By construction, the diagram
	\[
	\begin{tikzcd}
		\Hom_{N\mc}(\mathcal{ZW}_2^N(p,q),X)\ar[d,equal]\ar[rr,"f\mapsto \mathbf{q}_-f\circ s"] & &\Hom_{(N-1)\mc}(\mathcal{ZW}_1^{N-1}(p',q'), \mathbf{q_-}X)\ar[d,equal]\\
		ZW_2^{p,q}(X) \ar[rr,"(\eta_X)_*"]& &ZW_1^{p',q'}(\mathbf{q_-}X)
	\end{tikzcd}
	\]
	commutes, and therefore the map
	\[
	(\eta_X)_*: {'E}_2^{p,q}(X)\to {'E}_2^{p,q}(\mathbf{j}_-\mathbf{q}_-X)\cong {'E}_1^{p',q'}(\mathbf{q}_-X)={'E}_1^{p',q'}(\mathcal{ZW}_1^{N-1}(p',q'))
	\]
	takes the tautological generator to the tautological generator (recall \Cref{generator}). Thus by \Cref{generator} and \Cref{page-comparison}, $\eta_X$ induces an isomorphism on $'E_2$. As for $''E_1$, both
	\[
	''E_1(X),\  ''E_1(\mathbf{j}_-\mathbf{q}_-X)\cong {''E_1}(\mathbf{q}_-X)
	\]
	vanish by \Cref{computation}. This completes the proof for $X=\mathcal{ZW}_2^N(p,q)$.
	
	The proof for $X=\mathcal{ZW}_1^N(p,q)^{\mathfrak{o}_N}$ is simpler. On the one hand, ignoring degree shifts, both
	\[
	'E_2(X),\  'E_2(\mathbf{j}_-\mathbf{q}_-X)\cong {'E_1}(\mathbf{q}_-X)
	\]
	vanish. On the other hand, $\eta_X$ is an $''E_1$ isomorphism if and only if the unit of the adjunction $\mathbf{q}_+\dashv\mathbf{j}_+$, $\mathcal{ZW}_1^N(p,q)\to \mathbf{j}_+\mathbf{q}_+\mathcal{ZW}_1^N(p,q)$, is an $'E_1$-isomorphism. But the induced map
	\[
	'E_1(\mathcal{ZW}_1^N(p,q))\to {'E}_1(\mathbf{j}_+\mathbf{q}_+\mathcal{ZW}_1^N(p,q))={'E}_1(\mathcal{ZW}_1^{N-1}(p,q))
	\]
	clearly identifies the tautological classes. 
\end{proof}

\begin{proposition}\label{model-equivalence}
	For $N\ge 3$, the following Quillen adjunctions are Quillen equivalences:
	\begin{align*}
	& \mathbf{q}_+: (N\mc)_{0,1} \leftrightarrows \left((N-1)\mc\right)_{0,0} : \mathbf{j}_+,\\
	& \mathbf{q}_-: (N\mc)_{1,0} \leftrightarrows \left((N-1)\mc\right)_{0,0} : \mathbf{j}_-.
\end{align*}
\end{proposition}
\begin{proof}
	We prove the assertion for $\mathbf{q}_-\dashv\mathbf{j}_-$, the other follows by involutions. Now on the one hand it follows from \Cref{page-comparison} that $\mathbf{j}_-$ reflects weak equivalences. On the other hand every object of $((N-1)\mc)_{0,0}$ is fibrant, so the ordinary unit $\eta$ is the derived unit. Therefore by \cite[Corollary 1.3.16(c)]{hovey2007model} it suffices to prove that
\[
\eta_X: X\to \mathbf{j}_-\mathbf{q}_-X
\]
is a weak equivalence for every cofibrant object $X$.
	
	Let $I=\{\iota_2^N(p,q),\iota_1^N(p,q)^{\mathfrak{o}_N}\}_{p,q\in \ZZ}$ be the set of generating cofibrations of $(N\mc)_{1,0}$. Let $i: X\to Y$ be a single $I$-cell attachment, that is $Y$ is obtained from $X$ by pushout along a generating cofibration. Then $i$ is injective by \Cref{monomorphism} and $Z=\coker(i)$ is of the form $\mathcal{ZW}_2^{N}(p,q)$ or $\mathcal{ZW}_1^N(p,q)^{\mathfrak{o}_N}$. Consider the commutative diagram
	\[
	\begin{tikzcd}
		0\ar[r] & {'E}_1(X)\ar[r]\ar[d,"(\eta_X)_*"] & {'E}_1(Y)\ar[r]\ar[d,"(\eta_Y)_*"] & {'E}_1(Z)\ar[r]\ar[d,"(\eta_Z)_*"] & 0\\
		0\ar[r] & {'E}_0(\mathbf{q}_-X)\ar[r] & {'E}_0(\mathbf{q}_-Y)\ar[r] & {'E}_0(\mathbf{q}_-Z)\ar[r] & 0
	\end{tikzcd}
	\]
	The first row is exact by \Cref{cell-attachment} if $i$ is an $\iota_2^N$-cell; if $i$ is an $(\iota_1^N)^{\mathfrak{o}_N}$-cell then ${'E}_1(Z)=0$ and thus the first row is also exact. The second row is exact by \Cref{exactness}. The map $(\eta_Z)_*$ is a quasi-isomorphism by \Cref{unit-computation}. Therefore if $(\eta_X)_*$ is a quasi-isomorphism, then so is $(\eta_Y)_*$.
	
	Meanwhile consider the commutative diagram
	\[
	\begin{tikzcd}
		0\ar[r] & {''E}_0(X)\ar[r]\ar[d,"(\eta_X)_*"] & {''E}_0(Y)\ar[r]\ar[d,"(\eta_Y)_*"] & {''E}_0(Z)\ar[r]\ar[d,"(\eta_Z)_*"] & 0\\
		0\ar[r] & {''E}_0(\mathbf{q}_-X)\ar[r] & {''E}_0(\mathbf{q}_-Y)\ar[r] & {''E}_0(\mathbf{q}_-Z)\ar[r] & 0
	\end{tikzcd}
	\]
	Both rows are exact by \Cref{exactness}. The map $(\eta_Z)_*$ is a quasi-isomorphism by \Cref{unit-computation} and therefore again we have that if $(\eta_X)_*$ is a quasi-isomorphism, then so is $(\eta_Y)_*$.
	
	Therefore $\eta_Y$ is a weak equivalence whenever $\eta_X$ is. Since direct sums and filtered colimits are exact in $R$-modules and all the functors involved preserve them, a transfinite induction shows that $\eta_X$ is a weak equivalence for every $I$-cell complex $X$. Since every cofibrant object is a retract of an $I$-cell complex (\cite[Proposition 2.1.18]{hovey2007model}), and weak equivalences are closed under retracts, we conclude that $\eta_X: X\to \mathbf{j}_-\mathbf{q}_-X$ is a weak equivalence for every cofibrant object $X$.	This completes the proof.
\end{proof}

\begin{theorem}
	For $N\ge 3$ and $r,s\ge 0$, the following Quillen adjunctions are Quillen equivalences:
	\begin{align*}
	& \mathbf{q}_+: (N\mc)_{r,s+1} \leftrightarrows \left((N-1)\mc\right)_{r,s} : \mathbf{j}_+,\\
	& \mathbf{q}_-: (N\mc)_{r+1,s} \leftrightarrows \left((N-1)\mc\right)_{r,s} : \mathbf{j}_-.
\end{align*}
\end{theorem}
\begin{proof}
	We prove the assertion for $\mathbf{q}_-\dashv\mathbf{j}_-$, the other follows by involutions. Similar to the proof of \Cref{model-equivalence}, it suffices to show that $\eta_X: X\to \mathbf{j}_-\mathbf{q}_-X$ is a weak equivalence for every cofibrant object $X$.
	
	Observe that every $(N\mc)_{r+1,s}$ cofibration is an $\left(N\mc\right)_{1,0}$ cofibration because every trivial $\left(N\mc\right)_{1,0}$ fibration is a trivial $(N\mc)_{r+1,s}$ fibration. This follows from the lifting-property characterization of cofibrations \cite[Lemma~1.1.10]{hovey2007model}. Indeed, an isomorphism of a spectral sequence at a certain page induces an isomorphism on every later page, and isomorphisms are in particular surjective.
	
	Now let $X$ be cofibrant in $(N\mc)_{r+1,s}$, then it is cofibrant in $\left(N\mc\right)_{1,0}$ and therefore \Cref{model-equivalence} implies that the unit $X\to \mathbf{j}_-\mathbf{q}_-X$ is a weak equivalence in $(N\mc)_{1,0}$ and consequently a weak equivalence in $(N\mc)_{r+1,s}$.
\end{proof}

\begin{corollary}
	There is a Quillen equivalence
	\[
	\mathbf{q}: (4\mc)_{1,1}\leftrightarrows (2\mc)_{0,0}: \mathbf{j},
	\]
	in which $\mathbf{q}=\mathbf{q}_+\mathbf{q}_-$ and $\mathbf{j}=\mathbf{j}_-\mathbf{j}_+$.
\end{corollary}
\begin{proof}
	This follows by composing the two adjacent Quillen equivalences.
\end{proof}
We now compare this adjunction with the one defined in \cite[Section~4.1]{cirici2025model}. For a $4$-multicomplex $L$, set
\[
\overline L
=L/\bigl(\mathcal A_4\delta_0L+\mathcal A_4\delta_3L\bigr).
\]
Under the standard identification of extension of scalars along a quotient with the corresponding module quotient, $\mathbf q_-$ first imposes the relation $\delta_0=0$ and regrades by $\tau^{-1}(p,q)=(2p-q,p)$, while $\mathbf q_+$ then imposes the remaining relation $\delta_3=0$.  Hence
\[
(\mathbf q_+\mathbf q_-L)^{p,q}
\cong \overline L^{\,2p-q,p},
\]
with bicomplex differentials induced by $\delta_1^L$ and $\delta_2^L$. Conversely, for a bicomplex $M=(M,\delta_0,\delta_1)$, equations \eqref{j+} and \eqref{j-} give
\[
(\mathbf j_-\mathbf j_+M)^{p,q}=M^{q,\,2q-p},
\qquad
(\delta_0,\delta_1,\delta_2,\delta_3)
=(0,\delta_0^M,\delta_1^M,0).
\]
These are precisely the functors denoted by $\mathbf q$ and $\mathbf j$ in \cite[Section~4.1]{cirici2025model}. Thus, the adjunction above agrees canonically with the adjunction of \cite[Proposition~4.1.1]{cirici2025model}, and the corollary answers \cite[Remark~4.1.2]{cirici2025model}.

\begin{remark}
	The Quillen equivalence proved above completes the linear comparison between $4$-multicomplexes and bicomplexes for the spectral-sequence model structures considered here. A homotopy-theoretic comparison between almost complex and complex geometry, however, should also retain the multiplicative structure.
	
	Such a multiplicative refinement is related to a question of Sullivan, who asked whether rational homotopy theory provides further obstructions to the integrability of an almost complex structure \cite{angella2020almost}. On the bicomplex side, Stelzig's pluripotential homotopy theory provides a model-categorical framework for commutative algebra objects in bicomplexes \cite{stelzig2025pluripotential}. The underlying bicomplex model structure in Stelzig's theory is different from the one used here. Over $\mathbb C$, $(2\mc)_{0,0}$ is a right Bousfield localization of that model structure \cite{cirici2025model}. Nevertheless, the two classes of weak equivalences agree on locally bounded bicomplexes, and hence in particular on bounded bicomplexes \cite[Theorem~1.21]{stelzig2025pluripotential}.
	
	It is therefore natural to ask whether the category of commutative algebra objects in $4$-multicomplexes admits a suitable model structure and whether the resulting model category is Quillen equivalent to the corresponding model category in Stelzig's theory.
\end{remark}

\paragraph{Declaration on the use of artificial intelligence.}
OpenAI Codex (GPT-5 series, accessed July--August 2026) was used as an interactive tool to explore proof strategies, check intermediate algebraic arguments, identify potentially relevant references, improve the organization and language of the manuscript, and edit the \LaTeX{} source. The conceptual proof of \Cref{monomorphism} was suggested by OpenAI Codex (GPT-5.6 Sol) as a simplification of the author's earlier proof based on explicit computations. All suggestions were critically assessed and revised by the author.
\bibliographystyle{plain}
\bibliography{ref}

\Addresses
\end{document}